\documentclass[12pt,a4paper]{article}
\usepackage{amsmath, amsthm, amscd, amsfonts, amssymb, graphicx, color}
\usepackage{accents}
\usepackage[bookmarksnumbered, colorlinks, plainpages]{hyperref}

\makeatletter
\@namedef{subjclassname@2010}{%
	\textup{2010} Mathematics Subject Classification}
\makeatother
\usepackage[mathscr]{euscript}
\usepackage{latexsym}
\usepackage{float}
\usepackage{multicol,tabularx}
\usepackage[none]{hyphenat}
\usepackage{tikz}
\usepackage{graphics}
\usepackage{mathtools}

\usepackage{amsfonts}
\usepackage{amssymb}
\newtheorem{theorem}{Theorem}
\newtheorem{observation}{Observation}
\newtheorem*{conjecture*}{Conjecture}
\usepackage[english]{babel}
\usepackage{multido}
\usepackage[nomessages]{fp}
\usepackage[margin=2.5cm]{geometry}
\usepackage{enumitem}

\newtheorem*{lemma*}{Lemma}

\newtheorem*{coro*}{Corollary}
\theoremstyle{definition}

\newtheorem*{defn*}{Definition}

\newtheorem*{res*}{Theorem}
\numberwithin{equation}{section}
\begin{document}
\title{On a conjecture about the strong odd chromatic number of planar graphs}

\author {Arun J Manattu \footnote{Email: arunjmanattu@gmail.com},~ Athira Vinay \footnote{E-mail: athiravinay99@gmail.com}~ and Aparna Lakshmanan S. \footnote{E-mail: aparnals@cusat.ac.in, aparnaren@gmail.com}\\
Department of Mathematics\\	Cochin University of Science and Technology\\Cochin -
	22}
\maketitle
\begin{abstract}
A proper coloring of a graph $G$ is said to be a strong odd coloring of $G$, if for every vertex $v$ and every color $c$, either $c$ appears on an odd number of vertices in the neighborhood of $v$ or $c$ is absent in the neighborhood of $v$. The strong odd chromatic number of $G$ is defined as the smallest integer $k$ for which $G$ admits a strong odd coloring using $k$ colors. In this paper, we evaluate the strong odd chromatic number of join of cycles and empty graphs and one point union of graphs. Using these results, we construct infinite family of planar graphs that serves as counter examples to a recent conjecture regarding the upper bound of the strong odd chromatic number of planar graphs. \\
\noindent \line(1,0){395}


\noindent \line(1,0){395}
\end{abstract}

\maketitle
\section{Introduction}
A proper $k$-coloring of a graph $G$ is an assignment $\phi: V(G) \rightarrow \{1,2,\ldots,k\}$ such that no two adjacent vertices receive the same color. A recently introduced coloring concept called the odd coloring, which is a strengthening of proper coloring, has created a great amount of interest among researchers. Defined by Petruševski and Škrekovski \cite{Pet} in 2022, the odd coloring is a proper coloring such that for every non-isolated vertex $v$ in G, there exists a color that occurs an odd number of times in the neighborhood of $v$. A graph is odd $k$-colorable if it admits an odd coloring using $k$ colors. The odd chromatic number of a given graph $G$, denoted by $\chi_{\text{o}}(G)$, is the minimum number of colors in an odd coloring of $G$. Determining the odd chromatic number of a graph is in general NP-hard \cite{Caro 2} and it is NP-complete \cite{Ahn} even when restricted to the class of odd $k$-colorable bipartite graphs for $k\geq 3$. Caro et al.\cite{Caro 2} conjectured that an upper bound of the odd chromatic number of a graph $G$, with $\Delta(G) \geq 3$, is $\Delta(G)+1$ and Dai et al. \cite{Dai} proved that the bound is true asymptotically. The Four Color Theorem \cite{App, Rob}, though lacking an austere theoretical proof, guarantees that for any planar graph $G$, the chromatic number $\chi(G) \leq 4$. The quest for an analogous bound for the odd chromatic number of planar graphs immediately fails, since $\chi_{\text{o}}(C_5) =5$. Petruševski and Škrekovski \cite{Pet} proved that every planar graph is odd $9$-colorable and proposed the conjecture that every planar graph is odd $5$-colorable. The best upper bound known for a planar graph $G$ is $\chi_{\text{o}}(G) \leq 8$, proved by Petr and Portier \cite{Petr} in 2023. While considering classes of planar graphs, Miao et al. \cite{Miao} claimed in 2024 that every planar graph without adjacent $3$-cycles admits an odd-$7$ coloring and every triangle free planar graph without intersecting $4$-cycles admits an odd $5$-coloring. But recently, Pradhan et. al. \cite{Prad} pointed out a fundamental flaw contained in the proof in \cite{Miao} affecting the validity of their claims.\par

A stronger version of odd coloring is the strong odd coloring defined in 2024 by Kwon and Park \cite{Kwon}. A proper coloring of the vertices of $G$ is a strong odd coloring if for every vertex $v$ in $G$, every color $c$ present in the neighborhood of $v$ appears an odd number of times in the neighborhood of $v$, i.e., if $\chi_1,\chi_2,\ldots,\chi_k$ are the color classes of a proper coloring of $G$, then $N(v) \cap \chi_i$ is either empty or $|N(v) \cap \chi_i|$ is odd for every $i \in [k]$ and $v \in V(G)$. Every simple graph admits a strong odd coloring since using different colors for every vertex of $G$ is a trivial strong odd coloring of $G$. The strong odd chromatic number of $G$, denoted by $\chi_{\text{so}}(G)$, is the minimum $k$ such that $G$ admits a strong odd coloring with $k$ colors. For any subgraph $H$ of a graph $G$ we have $\chi(H) \leq \chi(G)$, but a similar result does not hold for the strong odd chromatic number. For instance, $\chi_{\text{so}}(C_4)$ = 4, but Figure \ref{C4 corona K1} is a supergraph of $C_4$ whose strong odd chromatic number is 2. Now, strong odd coloring being a proper coloring, it is a direct observation that chromatic number of a graph $G$ is always a lower bound for the strong odd chromatic number of $G$. In fact, we can further prove that for every graph $G$, there exists a graph $H$ with the property that $\chi_{\text{so}}(H) = \chi(G)$ and $G$ is an induced subgraph of $H$. Consider a proper coloring of $G$ using $\chi(G)$ colors and let $\chi_1,\chi_2,\ldots,\chi_k$ be the color classes. For every $v \in V(G)$ with $|N(v) \cap \chi_i|$ even, attach a pendant vertex adjacent to $v$ and color it with $i$. For each vertex $v$, the number of pendant vertices attached will be equal to the number of color classes in $G$ that appears on even number of vertices in the neighborhood of $G$ in the initial proper coloring of $G$. The resulting graph $H$ will be a supergraph of $G$ with strong odd chromatic number $\chi(G)$. Figure \ref{C4 corona K1} explicitly demonstrates the case when $G = C_4$. \par

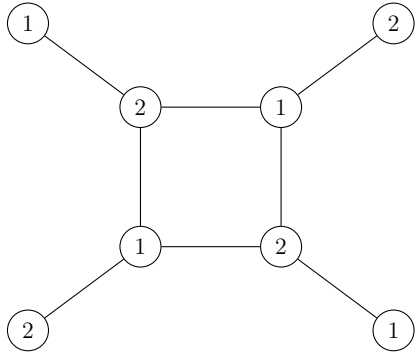
\begin{figure}[H] \label{C4 corona K1}
\centering
\begin{tikzpicture}[x=5cm, y=5cm,scale=0.74, transform shape]
         \node[circle, draw] (a1) at (0,0) {$1$};
\node[circle, draw] (a2) at (0,0.5)  {$2$};
\node[circle, draw] (a3) at (0.5,0.5) {$1$};
\node[circle, draw] (a4) at (0.5,0) {$2$};
\node[circle, draw] (b1) at (-0.4,-0.3)  {$2$};
\node[circle, draw] (b2) at (-0.4,0.8) {$1$};
\node[circle, draw] (b3) at (0.9,0.8) {$2$};
\node[circle, draw] (b4) at (0.9,-0.3) {$1$};

\draw (a1) edge (a2);
\draw (a2) edge (a3);
\draw (a3) edge (a4);
\draw (a4) edge (a1);

\draw (a1) edge (b1);
\draw (a2) edge (b2);
\draw (a3) edge (b3);
\draw (a4) edge (b4);

\end{tikzpicture}
\vspace{-0.3cm}
\caption{Construction of $H$ with $\chi_{\text{so}}(H) = \chi(C_4)$}
\end{figure}

While Kwon and Park \cite{Kwon} investigated the strong odd coloring of sparse graphs in their seminal paper, Michał Pilipczuk \cite{Michal} explored the same on graph classes of bounded expansion. Considering the $8$-odd colorability of planar graphs \cite{Petr}, Kwon and Park \cite{Kwon} posed the question whether the strong odd chromatic number of planar graphs is bounded above by some constant $c$. Caro, Petruševski, Škrekovski and Tuza \cite{Caro} proved that such a finite upper bound $c$ exists, which is 388. They also constructed two planar graphs both having strong odd chromatic number 12 and asked whether $\chi_{\text{so}}(G) \leq 12$ for every planar graph $G$. Pang, Miao and Fan \cite{Pang} recently gave a negative answer to this question by constructing a planar graph with strong odd chromatic number 13. In the same paper, they conjectured \cite{Pang} that every planar graph admits a strong odd 13-coloring. In this paper, we disprove this conjecture by producing infinite family of planar graphs with strong odd chromatic number 14. We also have examples of planar graphs for which strong odd chromatic number 15, 16 and 17.\par
  
In general, the exact values of the strong odd chromatic number of very few classes have been identified until this point. Similar to odd coloring \cite{Priya, Vida}, the strong odd coloring of product graphs has been predominantly explored for the Cartesian product, as well as for the lexicographic product \cite{Caro}. In this paper, we analyze the strong odd coloring of join of cycles and empty graphs. We also investigate the strong odd chromatic number of graphs formed by one point union of other graphs. This study eventually led to the construction of infinitely many planar graphs $G$ with $\chi_{\text{so}}(G) > 13$, thus disproving the conjecture \cite{Pang} posed by Pang et al. 
\section{Join of cycles and empty graphs}
\begin{defn*}
 The join $G = G_1 \vee G_2$ of two graphs $G_1$ and $G_2$ with vertex sets $V(G_1)$ and $V(G_2)$ is a graph formed by taking one copy each of $G_1$ and $G_2$ and making every vertex in $V(G_1)$ and $V(G_2)$ adjacent to each other.   
\end{defn*}

In this section, the strong odd chromatic number of join of cycles and empty graphs are determined. It maybe noted that, if there are at most 2 vertices in the empty graph under consideration, the resulting graph remains planar.  Now, while considering a coloring of a graph $G$ having a subgraph $H$, a color class $C$ is said to be an \textit{odd (respectively, even)} color class of $H$ if $|C \cap V(H)|$ is odd (even). Here, while considering the join of a cycle $C_n$ and an empty graph $\overline{K_m}$, the vertices in $\overline{K_m}$ being adjacent to all vertices in $C_n$, an even color class of $C_n$ would then contradict strong odd coloring of $C_n \vee \overline{K_m}$. Hence, the following observation, though direct, can be stated explicitly for convenience as we are using it repeatedly. 
\begin{observation} \label{H odd}
    In every strong odd coloring of $C_n \vee \overline{K_m}$, every color appearing in $C_n$ must be an odd color class of $C_n$.
\end{observation}

  A coloring of a graph $G$ is said to be a \textit{2-distance coloring}, if every pair of vertices at a distance at most 2 receive different colors \cite{Wegner}. Since, every vertex in the cycle is adjacent to exactly two vertices in the cycle and all the vertices in $\overline{K_m}$ are common to each vertex in $C_n$, the following observation can also be made. 

\begin{observation} \label{d>2}
    A strong odd coloring of $C_n \vee \overline{K_m}$, when restricted to the vertices in $C_n$ gives a 2-distance coloring.
\end{observation}

\begin{theorem} \label{wheels}
Let $W_n = C_n \vee K_1$ be the \textit{wheel graph} on $n+1$ vertices. Then,

    $\chi_{\text{so}}(W_n) = \begin{cases}
  4, & \text{if }~ n \equiv 3(mod~6)\\
  5,  & \text{if }~n \equiv 0, ~2 \text{ or }4(mod~6), ~n>8,~ n\neq15\\
  6,  & \text{if }~ n \equiv 1 \text{ or }5(mod~6), ~n>8\\
  7, & \text{if }~ n = 14\\
  n+1, & \text{if }~ n \leq 8\\  
\end{cases}$
\end{theorem}

\begin{proof}
If $n \leq 8$, by Observation \ref{d>2}, no color can appear more than 2 times in the cycle since among any three vertices of a cycle with at most 8 vertices, there is at least one pair with distance at most 2. But then, by Observation \ref{H odd}, this implies that every color appears exactly once, i.e; all vertices are colored with distinct colors. So, $\chi_{\text{so}}(C_n \vee K_1) = n+1$, for $ n\leq 8$. \par
For $n>8$, as a direct consequence of Observation \ref{d>2}, a color can be used at most $\lfloor \frac{n}{3} \rfloor$ times to color the vertices of $C_n$. Also, by applying Observation \ref{H odd}, a color can be used at most $\lfloor \frac{n}{3} \rfloor$ times, if $\lfloor \frac{n}{3} \rfloor$ is odd and $\lfloor \frac{n}{3} \rfloor - 1$ times, otherwise. So, when $n \equiv 0,1 \text{ or } 2\text{ (mod 6)}$, a color can be used at most $\lfloor \frac{n}{3} \rfloor - 1$ times to color $C_n$ and it can be used at most $\lfloor \frac{n}{3} \rfloor$ times, when $n \equiv 3,4 \text{ or } 5 \text{ (mod 6)}$.\par

When $n \equiv 0 \text{ (mod 6)}$, we need at least 4 colors to color $C_n$. Also, the colors $\{1,2,3,4\}$ can be sequentially used 3 times to color the first 12 vertices and the remaining vertices, if any, can be colored sequentially using the colors $\{1,2,3\}$.\par

When $n \equiv 1 \text{ or }5 \text{ (mod 6)}$, then again, we need at least 4 colors to color $C_n$. But, $n$ being odd and the cardinality of each color class being odd (by Observation \ref{H odd}), the number of colors must also be odd. So, we need at least 5 colors to color $C_n$. Now, for $n \equiv 1 \text{ (mod 6)}$, the colors $\{1,2,3,4\}$ can be sequentially used 3 times to color the first 12 vertices, color 5 to the next vertex and the remaining vertices, if any, can be colored sequentially using the colors $\{1,2,3\}$. For $n \equiv 5 \text{ (mod 6)}$, all vertices except the last two can be colored sequentially using colors $\{1,2,3\}$ and the last two vertices can be colored using 4 and 5.\par

When $n \equiv 2 \text{ (mod 6)}$,  the colors $\{1,2,3,4\}$ can be sequentially used 5 times to color the first 20 vertices (leaving $n=14$ as a special case) and the remaining vertices, if any, can be colored sequentially using the colors $\{1,2,3\}$. When $n=14$, since a color can be used at most 3 times $(
\lfloor\frac{14}{3}\rfloor - 1 = 3)$ to color $C_{14}$, we need at least 5 colors to color $C_{14}$. But, 14 being even, and by Observation \ref{H odd}, the number of colors used to color $C_{14}$ must be even. Therefore, we need at least 6 colors to color $C_{14}$. Now, the colors $\{1,2,3,4\}$ sequentially used 3 times followed by 5 and 6, gives the required coloring of $C_{14}$. \par

When $n \equiv 3 \text{ (mod 6)}$, the colors $\{1,2,3\}$ can be used sequentially $\frac{n}{3}$ times to color the vertices of $C_n$. When When $n \equiv 4 \text{ (mod 6)}$, then again, the colors $\{1,2,3\}$ can be used sequentially $\lfloor\frac{n}{3}\rfloor$ times to color the $n-1$ vertices of $C_n$ and the last vertex can be colored using 4.\par

In all these cases, $K_1$, being a universal vertex, needs a color different from the colors used to color $C_n$. Hence, the theorem.
\end{proof}

The coloring in the proof of the above theorem can be easily extended to $C_n \vee \overline{K_m}$. Since all the $m$ vertices of $\overline{K_m}$ are adjacent to all the vertices of $C_n$, the colors used for the vertices of $C_n$ cannot be given to the vertices of $\overline{K_m}$. As $V(\overline{K_m})$ is independent the same color can be used for all these vertices, if $m$ is odd. If $m$ is even, then $m-1$ vertices can be given the same color and the remaining vertex must be given a new color. Thus we have the following theorem. 

\begin{theorem} \label{2 centrals}
     $\chi_{\text{so}}(C_n \vee \overline{K_m}) = \begin{cases}
  \chi_{\text{so}}(W_n), & \text{if }~ n \text{ is odd}\\
   \chi_{\text{so}}(W_n) + 1, & \text{if }~ n \text{ is even}
\end{cases}$
\end{theorem}

\section{Strong odd chromatic number of planar graphs}
We begin this section by recalling the graph operation, one point union of graphs, and also define its generalization. This graph operation is of interest for us as it preserves planarity.

\begin{defn*} \cite{She}
   Let $x$ be a vertex in a graph $G$ and $n \in \textbf{N}$. The one point union of $n$ copies of $G$ about $x$, denoted by $I_x^n(G)$, is the graph obtained by identifying $x$ in $n$ copies of $G$. 
\end{defn*}
This calls for a natural generalization in the following way. 

 \begin{defn*}
     Let $G_i$ be a collection of graphs and let $x_i \in V(G_i)$, for $1\leq i \leq n$. The generalized one point union of $G_i$'s, about the vertices $x_i, ~1\leq i \leq n$, denoted by $I_{x}(G_1,G_2,\ldots,G_n)$ is the graph obtained by taking all $G_i$'s and identifying every vertex $x_i, ~1\leq i \leq n$ and renaming the new vertex as $x$. 
 \end{defn*}


\begin{theorem}
Let $G_n$ be the graph $C_n \vee \overline{K_2}$, for every $n \geq 3$ where $V(\overline{K_2}) = \{x_n,y_n\}$. Let $I_y(G_m,G_n)$ be the one point union of $G_m$ and $G_n$ obtained by identifying $y_m$ and $y_n$ and naming the new vertex as $y$. Then, $\chi_{\text{so}}(I_y(G_m,G_n)) = \chi_{\text{so}}(W_{m}) + \chi_{\text{so}}(W_n)-1$.
\end{theorem}

\begin{proof}
Let $C$ be a $\chi_\text{so}$-coloring of $I_y(G_m,G_n)$ with colors $c_1,c_2,\ldots,c_k$ and let $\chi_i$ be the color class with color $c_i$. Since, $x_m$ and $y$ are adjacent to all the vertices of $C_m$ and $C$ being a proper coloring, they cannot receive any color used to color $V(C_m)$. Also, since the vertices of $C_m$ do not have any neighbors outside $V(G_m)$, it follows that $C$ restricted to $C_m \vee \{x_m\} \equiv W_m$ is a strong odd coloring of $W_m$. The same arguments are applicable to $G_n$ also. Now, since $N(y) = V(C_m) \cup V(C_n)$ and $|\chi_i \cap V(C_m)|$ and $|\chi_i \cap V(C_n)|$ are both odd, the colors used to color $V(C_m)$ and $V(C_n)$ must be all distinct. Also, we can use a color used in $V(C_m)$ to color $x_n$ and a color used in $V(C_n)$ to color $x_m$. The vertex $y$ needs a color different from all these colors. Thus $|C| = \chi_{\text{so}}(W_{m}) + \chi_{\text{so}}(W_n)-1$.
\end{proof}
As we had already mentioned, it was conjectured in \cite{Pang} that every planar graph admits a strong odd 13-coloring. But, the above theorem in fact disproves this conjecture.

Applying Theorem \ref{wheels} in this result, we observe that $7 \leq \chi_{\text{so}}(I_y(G_m,G_n)) \leq 17$. In particular, we have\\
$\chi_{\text{so}}(I_y(G_7,G_{15})) = 14$,\\
$\chi_{\text{so}}(I_y(G_6,G_8)) = \chi_{\text{so}}(I_y(G_7,G_7)) = \chi_{\text{so}}(I_y(G_8,G_{15}))$ $ = 15$,\\
$\chi_{\text{so}}(I_y(G_7,G_8)) = 16$ and\\ $\chi_{\text{so}}(I_y(G_8,G_8)) = 17$.\par

Since $I_y(G_m,G_n)$ is planar for every $m, n \in \mathbb{N}$, these graphs serve as explicit counter examples for the conjecture posed by Pang et al. \cite{Pang}. In fact, for $n>9$ and $n\equiv 1~or ~5(mod~6)$, $\chi_{\text{so}}(I_y(G_8,G_n)) = 14$. This shows that there are infinitely many graphs serving as counter examples for the conjecture. As far as we have observed, with $\chi_{\text{so}}(I_y(G_8,G_8)) = 17$  (illustrated in Figure \ref{17}), $I_y(G_8,G_8)$ is the planar graph identified with the highest strong odd chromatic number.

\begin{figure}[H]
\centering
\begin{tikzpicture}[x=1.2cm, y=1.2cm,scale=0.7, transform shape]
         \node[circle, draw] (a1) at (0,3.5) {$1$};
\node[circle, draw] (a2) at (0,2.5)  {$2$};
\node[circle, draw] (a3) at (0,1.5) {$3$};
\node[circle, draw] (a4) at (0,0.5) {$4$};
\node[circle, draw] (a5) at (0,-0.5)  {$5$};
\node[circle, draw] (a6) at (0,-1.5) {$6$};
\node[circle, draw] (a7) at (0,-2.5) {$7$};
\node[circle, draw] (a8) at (0,-3.5) {$8$};
\node[circle, draw] (b1) at (-3,0) {$9$};
\node[circle, draw] (b2) at (4,0) {$17$};

\draw (a1) edge (a2);
\draw (a2) edge (a3);
\draw (a3) edge (a4);
\draw (a4) edge (a5);
\draw (a5) edge (a6);
\draw (a6) edge (a7);
\draw (a7) edge (a8);

\draw (a1) edge (b1);
\draw (a2) edge (b1);
\draw (a3) edge (b1);
\draw (a4) edge (b1);
\draw (a5) edge (b1);
\draw (a6) edge (b1);
\draw (a7) edge (b1);
\draw (a8) edge (b1);

\draw (a1) edge (b2);
\draw (a2) edge (b2);
\draw (a3) edge (b2);
\draw (a4) edge (b2);
\draw (a5) edge (b2);
\draw (a6) edge (b2);
\draw (a7) edge (b2);
\draw (a8) edge (b2);

\node[circle, draw] (c1) at (8,3.5) {$9$};
\node[circle, draw] (c2) at (8,2.5)  {$10$};
\node[circle, draw] (c3) at (8,1.5) {$11$};
\node[circle, draw] (c4) at (8,0.5) {$12$};
\node[circle, draw] (c5) at (8,-0.5)  {$13$};
\node[circle, draw] (c6) at (8,-1.5) {$14$};
\node[circle, draw] (c7) at (8,-2.5) {$15$};
\node[circle, draw] (c8) at (8,-3.5) {$16$};
\node[circle, draw] (d1) at (11,0) {$1$};

\draw (c1) edge (c2);
\draw (c2) edge (c3);
\draw (c3) edge (c4);
\draw (c4) edge (c5);
\draw (c5) edge (c6);
\draw (c6) edge (c7);
\draw (c7) edge (c8);

\draw (c1) edge (b2);
\draw (c2) edge (b2);
\draw (c3) edge (b2);
\draw (c4) edge (b2);
\draw (c5) edge (b2);
\draw (c6) edge (b2);
\draw (c7) edge (b2);
\draw (c8) edge (b2);

\draw (c1) edge (d1);
\draw (c2) edge (d1);
\draw (c3) edge (d1);
\draw (c4) edge (d1);
\draw (c5) edge (d1);
\draw (c6) edge (d1);
\draw (c7) edge (d1);
\draw (c8) edge (d1);

\draw[-] (a1) edge [out=180, in=180, looseness=2] (a8);
\draw[-] (c1) edge [out=0, in=0, looseness=2] (c8);

\end{tikzpicture}
\vspace{-0.5cm}
\caption{Planar graph $G$ with $\chi_{\text{so}}(G) = 17$}
\label{17}
\end{figure}
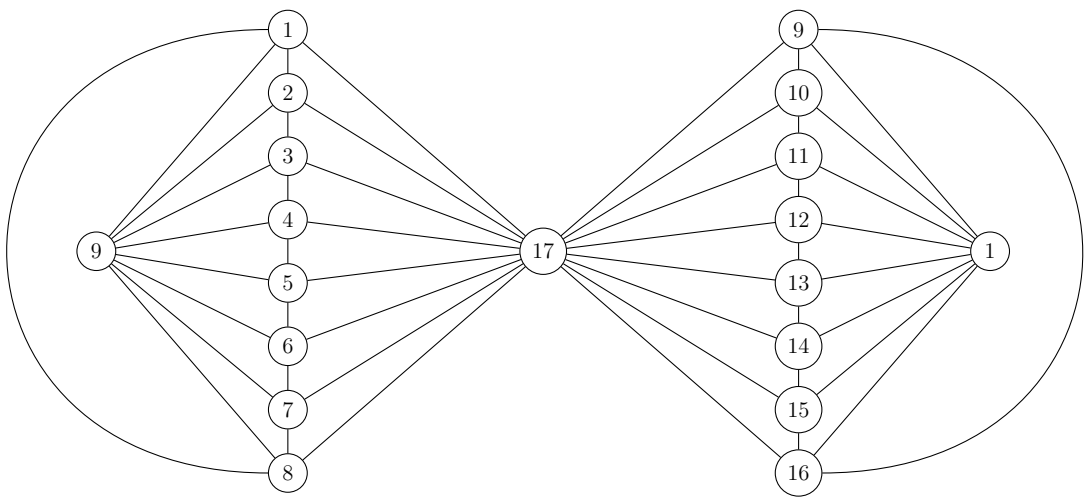

\section{Concluding remarks and future scope}
Strong odd coloring of graphs, though introduced very recently, has created attention and interest among researchers across the globe. Though most papers related to this concept revolves around bounds for classes of graphs, identifying the exact value of strong odd chromatic number of graph classes is still a good question. Though Cartesian and lexicographic products of graphs have been explored for their strong odd chromatic number \cite{Caro}, there are other graph products for which the same exposition be made. 

In this short note, we focused only on the conjecture posed in \cite{Pang}. Though we have identified some exact values and bounds for the corona operator, join and one point union of graphs during our search for a planar graph with strong odd chromatic number greater than 13, we have restricted this article to the infinite family of graphs which serves as a counter example to the conjecture. Although this paper establishes the existence of planar graphs $G$ with $\chi_\text{so}(G)=17$, we do not claim that this is the best possible. Rather we conclude by posing it as a problem to identify planar graphs with strong odd chromatic number greater than 17.

Another interesting aspect is that in \cite{Caro} they have given two planar graphs for which strong odd chromatic number is 12 and in \cite{Pang} only one planar graph with strong odd chromatic number 13 is given. But using our results, we can produce infinite families of graphs for which strong odd chromatic number is 12, 13 and 14. Though we are able to produce planar graphs with strong odd chromatic number 15, 16 and 17, we do not have a family of graphs is these cases. It may be true that there are only finitely many planar graphs for which strong odd chromatic number is greater than 14. We conclude this paper with the following problems.

\vspace{0.5cm}
\noindent{\bf Problem 1:} Identify planar graphs with strong odd chromatic number greater than 17.

\vspace{0.5cm}
\noindent{\bf Problem 2:} Find a constant $c < 388$ such that $\chi_\text{so}(G) < c$, for every planar graph $G$.

\vspace{0.5cm}
\noindent{\bf Problem 3:} Find a constant $c$ such that there are only finitely many planar graphs $G$ with $\chi_\text{so}(G) > c$.

\vspace{0.5cm}
\noindent \textbf{Acknowledgment:} The first author is supported by the Senior Research Fellowship (09/0239(17181)/2023-EMR-I) of CSIR (Council of Scientific and Industrial Research, India). The second author is supported by Cochin University of Science and Technology under the University Research Fellowship.

\end{document}